\documentclass{mhkr}

\usepackage[style = numeric, sorting = nyt]{biblatex}
\title{On a cofinal Reinhardt embedding without Powerset}
\author{Hanul Jeon}
\email{ \href{mailto:hj344@cornell.edu}{hj344@cornell.edu}}
\urladdr{ \href{https://hanuljeon95.github.io}{https://hanuljeon95.github.io} }
\address{Department of Mathematics, Cornell University, Ithaca, NY 14853} 
\thanks{The author would like to thank Asaf Karagila, who invited the author to the CHEESE\footnote{An acronym of \emph{CHoiceless Elementary Embeddings of SEts}} workshop from which the initial idea of this paper emerged. The author is also thankful to Yair Hayut for allowing the inclusion of the proof of \autoref{Proposition: Hayut's theorem for non-existence of a cofinal embedding}, and the reviewer for a meticulous and helpful review.}

\newcommand{\lag}{\langle}
\newcommand{\rag}{\rangle}

\newcommand{\calP}{\mathcal{P}}
\newcommand{\frakt}{\mathfrak{t}}

\newcommand{\sfQ}{\mathsf{Q}}

\newcommand{\crit}{\operatorname{crit}}
\newcommand{\Dom}{\operatorname{Dom}}
\newcommand{\Ran}{\operatorname{Ran}}
\newcommand{\trcl}{\operatorname{trcl}}
\newcommand{\rank}{\operatorname{rank}}

\newcommand{\Tcoll}{\operatorname{Tcoll}}
\newcommand{\Fin}{\operatorname{Fin}}
\newcommand{\suc}{\operatorname{suc}}
\newcommand{\lh}{\operatorname{lh}}

\newcommand{\ZFC}{\mathsf{ZFC}}
\newcommand{\ZF}{\mathsf{ZF}}
\newcommand{\DC}{\mathsf{DC}}

\begin{document}
\maketitle

\begin{abstract}
    In this paper, we provide a positive answer to the question of Matthews whether $\ZF^-$ is consistent with a non-trivial cofinal Reinhardt elementary embedding $j\colon V\to V$. The consistency follows from $\ZFC + I_0$, and more precisely, it is witnessed by Schlutzenberg's model of $\ZF$ with an elementary embedding $k\colon V_{\lambda+2}\to V_{\lambda+2}$.
\end{abstract}

\section{Introduction}

Set theorists developed large cardinal axioms as a gauge for the consistency strength of theories. 
Large cardinals stronger than measurable cardinals are usually described as critical points of elementary embeddings $j\colon V\to M$; We get a stronger large cardinal notion by demanding that $M$ is closer to $V$.
The previous description provided by Reinhardt \cite[\S 6.4]{Reinhardt1974RemarkReflection} has the natural culmination by equating $M=V$, which is now known as a \emph{Reinhardt embedding}:
\begin{definition}
    A \emph{Reinhardt embedding} is a non-trivial elementary embedding $j\colon V\to V$. The critical point of a Reinhardt embedding is called a \emph{Reinhardt cardinal}.
\end{definition}
The reader should be warned that $j$ is a proper class: The most typical way of understanding a proper class in first-order set theory like $\ZF$ or $\ZFC$ is a collection of sets satisfying a first-order formula with parameters in the language $\{\in\}$ of set theory.
However, Suzuki \cite{Suzuki1999} showed in $\ZF$ that no Reinhardt embedding is first-order definable, meaning that we cannot use the previous way to describe a Reinhardt embedding.
Hence, when we talk about a Reinhardt embedding, we expand the language of set theory by adding a unary function symbol $j$, and adding the axioms describing the elementarity of $j$ in the language of $\{\in\}$ with the axiom schemas of Separation and Replacement in the expanded language. (See \autoref{Subsection: Set theory} for more details.)

Unfortunately, Reinhardt's ultimate large cardinal axiom was doomed to an Icarian fate: Kunen proved that the existence of a Reinhardt embedding is inconsistent with $\ZFC$.
What Kunen proved is stronger than the mere inconsistency of a Reinhardt embedding.
\begin{definition}
    A cardinal $\lambda$ is a \emph{Kunen cardinal} if there is an elementary embedding $j\colon V_{\lambda+2}\to V_{\lambda+2}$.
\end{definition}

We can rephrase Kunen's result as follows:
\begin{theorem}[Kunen 1969, $\ZFC$]
    A Kunen cardinal does not exist. 
\end{theorem}
Kunen's theorem implies $\ZFC$ proves there is no Reinhardt cardinal: Suppose $j\colon V\to V$ is a Reinhardt embedding. If we let $\lambda = \sup_{n<\omega} j^n(\crit j)$, then $j\restriction V_{\lambda+2}$ witnesses $\lambda$ is a Kunen cardinal.

We can ask if the concept of a Reinhardt cardinal can be weakened to something that can consistently exist.  The following hypotheses provide possible weakenings in this spirit.
\begin{definition}
    Let us define the following assertions about a cardinal $\lambda$:
    \begin{itemize}
        \item ($I_3(\lambda)$) There is a non-trivial elementary embedding $j\colon V_\lambda \to V_\lambda$.
        \item ($I_2(\lambda)$) There is a non-trivial $\Sigma_1$-elementary embedding $j\colon V_{\lambda+1} \to V_{\lambda+1}$.
        \item ($I_1(\lambda)$) There is a non-trivial full elementary embedding $j\colon V_{\lambda+1} \to V_{\lambda+1}$.
        \item ($I_0(\lambda)$) There is a elementary embedding $j\colon L(V_{\lambda+1}) \to L(V_{\lambda+1})$.
    \end{itemize}
    For $0\le n\le 3$, $I_n$ is the assertion `For some $\lambda$, $I_n(\lambda)$ holds.'
\end{definition}
Solovay-Reinhardt-Kanamori \cite[\S 7]{SolovayReinhardtKanamori1978} defined the first three with the comment ``It seems likely that $I_1$, $I_2$ and $I_3$ are all inconsistent since they appear to differ
from the proposition proved inconsistent by Kunen only in an inessential technical
way,'' \cite[p. 109]{SolovayReinhardtKanamori1978}. However, as Kanamori \cite{Kanamori2008} stated ``These propositions have thus far defied all attempts at refutation in $\ZFC$'' \cite[\S 24]{Kanamori2008}, no known inconsistency proof of the first three is known.
The last principle $I_0(\lambda)$ was introduced by Woodin in 1984 to prove projective determinacy. Like the first three large cardinal notions, there is no known inconsistency proof of $I_0$ in $\ZFC$.

Alternatively, we may ask about the consistency of a Reinhardt embedding over a weaker subsystem of $\ZFC$. 
The most natural theory we can consider is $\ZF$, but it is not known if $\ZF$ with a Reinhardt cardinal is consistent. Surprisingly, the case of a Kunen cardinal is different modulo $\ZFC+I_0$:
\begin{theorem}[Schlutzenberg \cite{Schlutzenberg2020Kunen}] \pushQED{\qed} \label{Theorem: Schlutzenberg's theorem}
    If $\mathsf{ZFC}+I_0$ is consistent, then so is $\mathsf{ZF}$ with a Kunen cardinal. 
\end{theorem}

Another possibility we can consider is $\ZFC$ without Powerset, namely $\ZFC^-$. More precisely, $\ZFC^-$ is obtained from $\ZFC$ by removing Powerset but adding Collection. (See \autoref{Subsection: Set theory} to see why we add Collection.)
Unlike the full $\ZFC$ case, Matthews \cite{Matthews2020} proved that $\ZFC^-$ with a Reinhardt cardinal is consistent.
\begin{theorem}[Matthews {\cite{Matthews2020}}]
    If $I_1(\lambda)$ holds, then $H_{\lambda^+}$ is a model of $\ZFC^-_j$ with a non-trivial Reinhardt embedding $j$. Furthermore, $V_\lambda\in H_{\lambda^+}$ so $H_{\lambda^+}$ thinks $V_\lambda$ exists.
\end{theorem}
Here $\ZFC^-_j$ is the theory over the language of $\{\in,j\}$ with the axioms of $\ZFC$ without Powerset, but with Collection and Separation in the expanded language. The assertion `$j$ is a non-trivial elementary embedding' is the combination of the following two assertions:
\begin{itemize}
    \item (Nontriviality) There is $x$ such that $j(x)\neq x$.
    \item (Elementarity) For every formula $\phi(x_0,\cdots, x_{n-1})$ in the language $\{\in\}$, we have
    \begin{equation*}
        \forall x_0\cdots\forall x_{n-1} [\phi(x_0,\cdots, x_{n-1}) \lr \phi(j(x_0),\cdots,j(x_{n-1}))]. \footnotemark
    \end{equation*}
\end{itemize}
\footnotetext{Elementarity is an axiom schema and not a single axiom.}%
The main idea of Matthews' proof \cite{Matthews2020} for obtaining a model of $\ZFC^-_j$ with a Reinhardt embedding $j\colon V\to V$ is turning $V_{\lambda+1}$ for an $I_1$-cardinal $\lambda$ into $H_{\lambda^+}$.
We can understand Matthews' proof as follows: We can view $V_{\lambda+1}$ as a model of second-order set theory by understanding $V_{\lambda+1}$ as a universe $V_\lambda$ together with its classes $\calP(V_\lambda) = V_{\lambda+1}$. 
Then we not only have $V_{\lambda+1}\subseteq H_{\lambda^+}$, but also that every member $X$ of $H_{\lambda^+}$ is coded by a relation $E\subseteq \lambda\times\lambda$ encoding $(\trcl X, \in)$. 
Then we can extend $j\colon V_{\lambda+1}\to V_{\lambda+1}$ into $k\colon H_{\lambda^+}\to H_{\lambda^+}$ by defining $k(X)$ to be the `decoding' of $j(E)$ when $E\subseteq \lambda\times\lambda$ encodes $X$.
However, the Reinhardt embedding over the Matthews' model fails to satisfy the following property, unlike a Reinhardt embedding over $\ZF$:
\begin{definition}
    An elementary embedding $j\colon V\to V$ is \emph{cofinal} if for every $a$ there is $b$ such that $a\in j(b)$.
\end{definition}
Thus, the following question arises:
\begin{question}
    Is the following theory consistent: $\ZF^-_j$ with a \emph{cofinal} Reinhardt elementary embedding $j\colon V\to V$?
\end{question}

The main goal of this paper is to answer the above question positively, namely:
\begin{theorem*}
    $\ZFC + I_0$ proves there is a transitive model of $\ZF^-_j$ with a cofinal Reinhardt embedding $j$.  
\end{theorem*}

To prove the above theorem, we apply the same idea to get a model of $\ZFC^-_j$ with a Reinhardt embedding, but we need a model of second-order set theory different from $V_{\lambda+1}$ since the $I_1$ embedding becomes a non-cofinal elementary embedding over the universe of coded sets $H_{\lambda^+}$. We overcome this issue by turning $V_{\lambda+2}$ with an elementary embedding $j\colon V_{\lambda+2}\to V_{\lambda+2}$ into a model of first-order set theory. Although Kunen cardinals are inconsistent with $\ZFC$, a recent result by Schlutzenberg \cite{Schlutzenberg2020Kunen} states that $\ZFC+I_0$ interprets $\ZF$ with a Kunen cardinal, so we can use $V_{\lambda+2}$ for a Kunen cardinal $\lambda$ to construct a desired model.
We will work in Schlutzenberg's model satisfying $\ZF$ with the existence of a Kunen cardinal $\lambda$, and consider the collection of sets coded by a member of $V_{\lambda+2}$. It will turn out that the collection is a transitive model of $\ZF^-_j$ with a cofinal Reinhardt embedding $j$. 

There are at least two ways to code sets on the first-order side to classes on the second-order side: 
Matthews \cite{Matthews2020} codes a set $a$ into its \emph{membership code} $(\trcl(a),\in\restricts\trcl(a))$ akin to Williams' approach \cite[Chapter 2]{WilliamsPhD} and that of Goldberg \cite{Goldberg2021EvenOrdinals}. Anton and Friedman \cite{AntosFriedman2017Hyperclass} coded sets as well-founded trees, which is close to what Simpson did in \cite[\S VII.3]{Simpson2009}. 
Both approaches have pros and cons: The membership code approach allows us to represent sets in a unique membership code up to isomorphism. However, constructing a new membership code from previous ones requires a technical maneuver. Constructing a new set is easier in the tree approach, but different non-isomorphic trees can represent the same set. In this paper, we choose the tree approach that Simpson used in \cite{Simpson2009}.

\section{Preliminaries}

\subsection{Set theory and elementary embeddings} \label{Subsection: Set theory}
In this subsection, we clarify basic notions about $\ZFC$ without Powerset and elementary embeddings. Formulating set theory without Powerset requires more caution than we may naively expect since $\ZFC$ without Powerset exhibits various ill-behaviors. For example, \cite{GitmanHamkinsJohnstone2016} proved that it is consistent with $\ZFC$ without Powerset that $\omega_1$ exists but is singular. It turns out that replacing Replacement with Collection avoids these ill-behaviors.
We also must be careful to formulate Choice over $\ZF$ without Powerset, since equivalent formulations of Choice over $\ZF$ become non-equivalent without Powerset: 
Szczepaniak proved \cite[Theorem III]{Zarach1983UnionZFmodels} that the Choice of the form `every family of non-empty sets has a choice function' does not imply the Well-ordering principle if we do not have Powerset, and even adding Collection does not resolve the non-implication.
Hence, when we formulate `$\ZFC^-$ without Powerset,' we use the Well-ordering principle instead of Choice.

\begin{definition} \label{Definition: ZFC without Powerset}
    $\ZF^-$ is $\ZF$ without Powerset but with the Collection Scheme instead of the Replacement Scheme.
    $\ZFC^-$ is obtained from $\ZF^-$ by adding the \emph{Well-ordering principle}, which asserts that every set is well-orderable.
\end{definition}

Then let us define $\ZF^-_j$ as follows:
\begin{definition}
    $\ZF^-_j$ is the theory over the language $\{\in,j\}$, whose axioms are the usual axioms of $\ZF^-$ 
    with Nontriviality and Elementarity, with  Separation and Collection schema in the new language $\{\in,j\}$.
    $\ZFC^-_j$ is obtained from $\ZF^-_j$ by adding the Well-ordering principle. 
    We often add the expression `with a Reinhardt embedding' after $\ZF^-_j$ to stress that $j\colon V\to V$ is a Reinhardt embedding.
\end{definition}

As stated in the Introduction, Matthews \cite{Matthews2020} proved that a Reinhardt embedding is compatible with $\ZFC^-$.
Matthews also proved a form of Kunen inconsistency theorem in the following form:
\begin{theorem}[Matthews {\cite[Theorem 5.4]{Matthews2020}}] \label{Theorem: Matthew's Kunen inconsistency}
    Working over $\ZFC^-_j$, if $V_\lambda$ exists for $\lambda=\sup_{n<\omega}j^n(\crit j)$,%
    \footnote{$\ZF^-_j$ proves $\lambda$ exists since $\lag j^n(\crit j)\mid n<\omega\rag$ exists by Replacement in the expanded language.}
    then $j\colon V\to V$ is not cofinal. 
\end{theorem}
Note that Matthews proved \autoref{Theorem: Matthew's Kunen inconsistency} from the $\Sigma_0$-elementarity of $j$, not the full elementarity.
However, we can check that a cofinal $\Sigma_0$-elementary $j\colon V\to V$ is also fully elementary.
The additional assumption `$V_\lambda$ exists' in \autoref{Theorem: Matthew's Kunen inconsistency} does not follow from the existence of a Reinhardt embedding by the following result:
\begin{theorem}[Gitman-Matthews {\cite[Theorem 6.4]{GitmanMatthews2022}}]
    There is a model of $\ZFC^-_j$ with a Reinhardt embedding plus `$\mathcal{P}(\omega)$ is a proper class.'
\end{theorem}

Yair Hayut proved that the extra assumption `$V_\lambda$ exists' in \autoref{Theorem: Matthew's Kunen inconsistency} is unnecessary.
The proof is included with Hayut's permission:
\begin{proposition} \label{Proposition: Hayut's theorem for non-existence of a cofinal embedding}
    Working over $\ZFC^-_j$, $j\colon V\to V$ cannot be cofinal.
\end{proposition}
\begin{proof}
    Suppose $j\colon V\to V$ is a cofinal Reinhardt embedding, and let $\kappa_n = j^n(\crit j)$.
    By \cite[Theorem 5.2]{Matthews2020}, $\lambda^+$ for $\lambda = \sup_{n<\omega} \kappa_n$ cannot exist, so every set surjects onto $\lambda$ since every set is well-orderable and $\lambda$ is the largest cardinal. We also note that $j(\lambda) = \lambda$. Since $j$ is cofinal, we can find $z$ such that $j^"[\lambda]\in j(z)$. Fix an onto map $f\colon \lambda\to z$ and $\zeta<\lambda$ such that $j(f)(\zeta)=j^"[\lambda]$.
    
    Now fix $n$ such that $\zeta<\kappa_n$ and let $\mathcal{U} = \{X\subseteq \kappa_n \mid \zeta\in j(X)\}$ be the induced ultrafilter over $\kappa_n$. If we take
    \begin{equation*}
        A_\alpha = \{\beta < \kappa_n \mid \alpha \in f(\beta)\},
    \end{equation*}
    then $A_\alpha\in\mathcal{U}$ for every $\alpha<\lambda$.
    Note that $\mathcal{U}$ itself could be a proper class, but the family $\mathcal{U}_0=\{A_\alpha\mid \alpha<\lambda\}$ is a set. Now let us claim the following: There is $B\subseteq \kappa_n$ such that 
    \begin{equation*}
        E = \{\alpha<\lambda\mid A_\alpha = B\}
    \end{equation*}
    has at least $\kappa_{n}$ elements. It is easy to prove if $V_{\kappa_0}$ exists: Then because $\kappa_0$ is inaccessible, so is $\kappa_{n}$, and $\mathcal{U}\in V_{\kappa_{n}+1}$ has cardinality less than $\kappa_{n+1}$. Since we have $\lambda$ many ordinals indexing $A_\alpha$, we can apply the pigeonhole principle. However, we do not know whether $V_{\kappa_{n+1}}$ exists, and the inaccessibility of $\kappa_{n+1}$ is also unclear. Despite that, we can prove $\kappa_{n+1}$ is close to an inaccessible cardinal:

    \begin{lemma} \label{Lemma: Semi inaccessibility of a critical point}
        Let $\mathcal{A}$ be a set whose rank is less than $\kappa_0$. Then there is no surjection $g\colon \mathcal{A}\to \kappa_0$.
    \end{lemma}
    \begin{proof}
        First, we can prove that for every set $a$ with rank less than $\kappa_0$, we have $j(a) = a$. We prove it by induction on $\rank a < \kappa_0$: Suppose that for every $x$ with $\rank x<\rank a$, we have $j(x)=x$. We claim that $j(a)\cap V_{\kappa_0}\subseteq a$.%
        \footnote{Here let $V_{\kappa_0}$ be the class of all sets whose rank is below $\kappa_0$. It can be a proper class but for each set $a$, $a\cap V_{\kappa_0}$ is a set by Full Separation.}
        For $y\in j(a)$ such that $\rank y<\kappa_0$, we have $\rank j(y) = j(\rank y) = \rank y < \rank j(a)$, so $\rank y < \rank a$. Hence, by the inductive hypothesis, we have $j(y)=y$, so $y\in a$. This shows $j(a)\cap V_{\kappa_0}\subseteq a$.
        
        Thus, if $j(a)\setminus a$ is not empty, then $j(a)\setminus a$ must have an element of rank at least $\kappa_0$. This implies $j(\rank a) = \rank j(a) > \kappa_0$, contradicting with that $\rank a < \kappa_0$.

        Now suppose that there is a surjection $g\colon\mathcal{A}\to\kappa_0$. Since $\mathcal{A}$ has rank less than $\kappa_0$, we have $j(\mathcal{A})=\mathcal{A}$.
        Hence $g=j(g)\colon\mathcal{A}\to j(\kappa_0)$ is also a surjection, a contradiction.
    \end{proof}
    As a corollary, if $\mathcal{A}$ is a set whose rank is less than $\kappa_{n+1}$, then there is no surjection from $\mathcal{A}$ to $\kappa_{n+1}$.

    Let us go back to the main proof.
    We claim that there is $B\in \mathcal{U}_0$ such that $E=\{\alpha<\lambda\mid A_\alpha = B\}$ has at least $\kappa_n$ elements.
    Assume the contrary that for each $A_\alpha$ there are fewer than $\kappa_{n}$ many $\beta<\lambda$ such that $A_\alpha=A_\beta$. 
    To derive the contradiction, let us consider the map $g\colon \kappa_n\times \mathcal{U}_0 \to \lambda$ given by
    \begin{equation*}
        g(\alpha,B) =
        \begin{cases}
            \gamma & \text{If }B=A_\gamma,\ \operatorname{ordertype}\{\delta<\gamma\mid A_\delta=B\} = \alpha, \\
            0 & \text{Otherwise.}
        \end{cases}
    \end{equation*}
    The idea of the definition of $g$ is as follows: For each $B\in \mathcal{U}_0$, consider the set $E_B = \{\gamma<\lambda\mid A_\gamma = B\}$. Then $|E_B| < \kappa_n$ for each $B\in\mathcal{U}_0$. Then $g$ maps $(\alpha,B)$ into the $\alpha$th element of $E_B$ if exists. Since $\lambda = \bigcup_{B\in\mathcal{U}_0} E_B$ and each $E_B$ has size less than $\kappa_n$, $g$ is surjective.
    In particular, we found a surjective map from a set rank less than $\kappa_n\times\mathcal{U}_0$ to $\lambda$, so we also have a surjective from a set rank less than $\kappa_n\times\mathcal{U}_0$ to $\kappa_{n+1}$.
    This is impossible by \autoref{Lemma: Semi inaccessibility of a critical point} since the rank of $\kappa_n\times\mathcal{U}_0$ is less than $\kappa_{n+1}$.

    Thus we can pick $B\in \mathcal{U}_0$ such that $E=\{\alpha<\lambda\mid A_\alpha = B\}$ has at least $\kappa_n$ elements. By definition of $A_\alpha$ and $E$, we have
    \begin{equation*}
        \forall \beta\in B (f(\beta)\supseteq E).
    \end{equation*}
    Also, $B\in \mathcal{U}$ since $B$ is one of $A_\alpha$, so $\zeta\in j(B)$.
    Hence we have $j(E)\subseteq j(f)(\zeta)=j^"[\lambda]$.
    However, $E$ has at least $\kappa_n$ elements, so the $\kappa_0$th element of $j(E)$ is not in the range of $j$.\footnote{If $j(\gamma)\in j(E)$ is the $\kappa_0$th element of $j(E)$, then by elementarity, $\gamma\in E$ is the $\xi$th element of $E$ for some $\xi<\kappa$. However, $j(\xi)=\xi<\kappa_0$, a contradiction.} That is, $j(E)\nsubseteq j^"[\lambda]$, a contradiction.
\end{proof}

\subsection{\texorpdfstring{$H_\lambda$}{H lambda} and \texorpdfstring{$H(X)$}{H(X)}}
The sets $H_\lambda$ consisting of all sets of hereditary cardinality less than $\lambda$ forms an important class of models of second-order $\ZFC^-$ for regular cardinals $\lambda$.
However, the definition and properties of $H_\lambda$ rely on the axiom of choice. In a choiceless world, there is no reason to believe that $\ZFC$-equivalent definitions of $H_\lambda$ are still equivalent. Here is a list of possible $\ZF$-definitions of $H_\lambda$ in a choiceless context:
\begin{enumerate}
    \item (\cite[Definition 2.8]{Goldberg2021EvenOrdinals}) $H(X)$ is the union of all transitive sets $M$ such that $M$ is an image of some $S\in X$.

    \item (\cite[Definition 1.7]{LubarskyRathjen2003}) $H^+(X)$ is the smallest class $Y$ such that whenever $f\colon S\to Y$ for some $S\in X$, then $\ran f\in Y$.

    \item (\cite[Definition 2.1]{AsperoKaragila2021}) $\mathcal{H}(X)$ is the class of sets $y$ such that there is no surjection from $\trcl(y)$ to $X$.
\end{enumerate}
We take the first definition for $H_\lambda$ in a choiceless context.
Also, \emph{we work over $\ZF$ throughout this subsection.}

It is clear that $\ZFC$ proves $H(\lambda)=H_\lambda$.
But it is not immediately clear that $\ZF$ proves $H(X)$ is a set from its definition. We prove that $\ZF$ proves $H(X)$ is always a set. For a set $S$, let us define $\theta(S)$ to be the supremum of all ordinals that are images of $S$; i.e.,
\begin{equation*}
    \theta(S) := \sup\{\alpha\in\mathrm{Ord}\mid \exists f\colon S\to \alpha (\text{$f$ is onto})\}.
\end{equation*}
Note that if $\theta(S)$ exists and there is an onto map $f\colon S\to\alpha$, then $\alpha<\theta(S)$: We also have an onto map $g\colon\alpha\to(\alpha+1)$, so $g\circ f$ witnesses $\alpha+1 \le \theta(S)$. In particular, $\theta(S)$ is a limit ordinal if it exists.

\begin{lemma} \label{Lemma: Transitive image of S is bounded by theta S}
    $\theta(S)$ exists.
    Furthermore, if $M$ is a transitive set that is an image of $S$, then $\rank M < \theta(S)$.
\end{lemma}
\begin{proof}
    For the well-definability of $\theta(S)$, consider the set $\mathcal{W}\subseteq \mathcal{P}(S\times S)$ of prewellorders over $S$. For $\prec\in\mathcal{W}$, let $\lvert\prec\rvert$ be the supremum of the range of $\prec$-rank $\rho_\prec(x)$ recursively defined by $\rho_\prec(x) = \sup\{\rho_\prec(y)+1\mid y \prec x\}$. 
    Then let us take $\theta = \sup\{\lvert\prec\rvert : {\prec}\in\mathcal{W}\}$.
    For a given $f\colon S\to\alpha$, if we define a binary relation $\prec$ over $S$ by
    \begin{equation*}
        x \prec y \iff f(x)<f(y),
    \end{equation*}
    then ${\prec}\in \mathcal{W}$ and $\lvert\prec\rvert = \alpha$. Hence $\theta(S)\le \theta$. In fact, the reader can see that $\theta(S)=\theta$, which is unnecessary for the proof.
    
    For the remaining part, suppose there is a surjection $f\colon S\to M$. The transitivity of $M$ implies $\rank \colon M\to \rank(M)$ is onto, so $\rank \circ f\colon S\to \rank(M)$ is a surjection. Hence $\rank M < \theta(S)$.
\end{proof}

\begin{proposition} \label{Proposition: Properties of H(X)}
    Let $X$ be a non-empty set.
    \begin{enumerate}
        \item $H(X)$ is a transitive set.
        
        \item If $b\subseteq \trcl(a)$ and $a\in H(X)$, then $b\in H(X)$. 

        \item $H(X)$ is closed under subsets of its elements and union.
        
        \item Suppose that $X$ is closed under disjoint unions in the sense that there is a function $\sqcup\colon X\times X\to X$ such that for each $a,b\in X$, there are injective maps $\iota^0_{a,b}\colon a\to a\sqcup b$ and $\iota^1_{a,b}\colon b\to a\sqcup b$ satisfying $\ran \iota^0_{a,b}\cap \ran \iota^1_{a,b}=\varnothing$ and $\ran \iota^0_{a,b}\cup \ran \iota^1_{a,b} = a\sqcup b$.
        If we also have $1 = \{0\}\in X$, then $H(X)$ is closed under $a,b\mapsto \{a,b\}$.
    \end{enumerate}
\end{proposition}
\begin{proof}
    \begin{enumerate}
        \item By \autoref{Lemma: Transitive image of S is bounded by theta S}, we have $ H(X) \subseteq V_{\sup\{\theta(S)\mid S\in X\}}$. $H(X)$ is a union of transitive sets, so it is also transitive. 
        
        \item Suppose that $a\in M$ for some transitive $M$ such that $M$ is a surjective image of $S\in X$. Fix $b\subseteq\trcl(a)$, a surjection $f\colon S\to M$, and $x_0\in S$ such that $f(x_0)=a$. Observe that the transitivity of $M$ implies $\trcl(\{a\})\subseteq M$, so $\trcl (\{b\}) \subseteq \ran f=M$.
        Then define $g\colon S\to \trcl(\{b\})$ as follows: 
        \begin{equation*}
            g(x) =
            \begin{cases}
                f(x) & \text{if } f(x)\in\trcl(b),\\
                b & \text{if } x=x_0,\\
                0 & \text{otherwise}.
            \end{cases}
        \end{equation*}
        Then the image of $g$ is $\trcl(\{b\})$, so $b\in \trcl(\{b\})\subseteq H(X)$.
        
        \item Clearly if $b\subseteq a\in H(X)$ then $b\subseteq \trcl(a)$, so $b\in H(X)$. Also, $\bigcup a\subseteq \trcl(a)$ implies $\bigcup a\in H(X)$.
        
        \item Let $a,b\in H(X)$, $a\in M_0$, $b\in M_1$ for some transitive sets $M_0$ and $M_1$ such that there are surjections $f_0\colon S_0\to M_0$, $f_1\colon S_1\to M_1$ for some $S_0,S_1\in X$.
        By the assumption, $(S_0\sqcup S_1)\sqcup \{0\}\in X$.
        Let us define the surjection $g\colon (S_0\sqcup S_1)\sqcup \{1\} \to \trcl(\{\{a,b\}\})$ as follows:
        \begin{equation*}
            g(y) = 
            \begin{cases}
                f_0(x) & \text{if } y = \iota^0_{S_0\sqcup S_1, 1}(\iota^0_{S_0,S_1}(x)) \text{ for some }x\in S_0,\\
                f_1(x) & \text{if } y = \iota^0_{S_0\sqcup S_1, 1}(\iota^1_{S_0,S_1}(x)) \text{ for some }x\in S_0\\
                \{a,b\} & \text{if }y=\iota^1_{S_0\sqcup S_1, 1}(0), \\
                0 & \text{otherwise}.
            \end{cases}
        \end{equation*}
        Thus $\trcl(\{\{a,b\}\})$ is an image $g$ whose domain is a member of $X$. Hence $\{a,b\}\in \trcl(\{\{a,b\}\}) \subseteq H(X)$.
        \qedhere 
    \end{enumerate}
\end{proof}

The following is immediate from the previous proposition:
\begin{corollary} \label{Corollary: H(X) models SO Zermelo} \pushQED{\qed}
    Suppose that $X$ is closed under disjoint unions and $1\in X$. Then $H(X)$ is a model of Second-order Zermelo set theory without Powerset and Choice.
    That is, $H(X)$ satisfies Extensionality, Foundation, Pairing, Union, Infinity, and the Second-order Separation.
    \qedhere 
\end{corollary}

We do not know if $H(X)$ satisfies Second-order Replacement or Collection in general: Working in $\ZFC$, $M = H(\{\omega_n\mid n<\omega\}) = \{x : \lvert\trcl(x)\rvert < \omega_\omega\}$ thinks each $\omega_n$ is in $M$, but $\lag \omega_n\mid n<\omega\rag\notin M$.
In a choiceless context, even if $H(X)$ satisfies Second-order Replacement, it may not satisfy Second-order Collection:

\begin{example}
    Let us work over $\ZF$ + ``$\mathbb{R}$ is a countable union of countable sets.'' 
    \cite[Proposition 4.2]{LubarskyRathjen2003} proved that $H^+(\omega+1)$, which is the set of hereditarily countable sets, does not satisfy Second-order Collection even though it satisfies Second-order Replacement.
    Now we claim that $H(H^+(\omega+1)) = H^+(\omega+1)$:
    First, observe that every member of $H^+(\omega+1)$ is an image of $\omega$, and $H^+(\omega+1)$ is a transitive model of Second-order Zermelo set theory. Hence $H(H^+(\omega+1))\supseteq H^+(\omega+1)$. 
    Now suppose that $M$ is a transitive set and $f\colon S\to M$ for some $S\in H^+(\omega+1)$. Since there is an onto map from $\omega$ to $S$, we have a surjective map $h\colon\omega\to M$.
    Then we prove by induction on $a\in M$ that $a\in H^+(\omega+1)$.
    Suppose that $a\in M$ and $a\subseteq H^+(\omega+1)$. Fix $x_0\in a$, and define
    \begin{equation*}
        g(x) = 
        \begin{cases}
            h(x) & h(x)\in a, \\
            x_0 & \text{otherwise.}
        \end{cases}
    \end{equation*}
    Then $g\colon \omega\to a$ is onto, so $a\in H^+(\omega+1)$.
\end{example}
Second-order Collection over $H(X)$ usually follows from choice-like principles. We will see a case when $H(X)$ satisfies Second-order Collection in \autoref{Section: Schlutzenberg model}.

\subsection{Flat pairing, flat finite sequences, and flat product}
The usual Kuratowski ordered pair raises the rank by 2, that is, the rank of $\lag a,b\rag := \{\{a\},\{a,b\}\}$ is $\max(\rank a,\rank b)+2$. That is, $a,b\in V_\lambda$ does not guarantee $\lag a,b\rag\in V_\lambda$ unless $\lambda$ is limit. For our arguments, we will need
ordered pairs and finite tuples of elements of $V_{\lambda+2}$ inside $V_{\lambda+2}$.
In this subsection, we introduce a variant of the Quine-Rosser ordered pair such that every $V_\lambda$ for $\lambda>\omega$ is closed under its formation.

Let $s$ be a function defined by $s(x)=2x+1$ if $x\in\omega$, and $s(x)=x$ otherwise. Then define $f_0,f_1\colon V\to V$ by $f_0(x) = s^"[x]$, $f_1(x)=s^"[x]\cup\{0\}$. We can observe that $f_0$ and $f_1$ are injective, and $f_0^"[a]\cap f_1^"[b]=0$ for every $a$ and $b$, so we can code a pair of $a$ and $b$ as a single set by taking the union of $f_0^"[a]$ and $f_1^"[b]$: 
\begin{definition}
    For two sets $a$ and $b$, define $\lceil a,b\rceil = f_0^"[a]\cup f_1^"[b]$.
\end{definition}

Then the following is immediate:
\begin{lemma} \label{Lemma: flat pairing properties}
    Let $\lambda\ge \omega$ be an ordinal.
    \begin{enumerate}
        \item $a,b\mapsto \lceil a,b\rceil$ is an injection from $V_\lambda\times V_\lambda$ to $V_\lambda$. 
        \item The function $a,b\mapsto \lceil a,b\rceil$ and its projection functions are all $\Delta_0$-definable over $V_\lambda$ with parameter $\omega$.
        \item $\rank \lceil a,b\rceil$ is either finite or $\rank \lceil a,b\rceil \le \max(\rank a,\rank b)$. Especially, if $\rank a,\rank b<\lambda$ then $\rank \lceil a,b\rceil < \lambda$.
    \end{enumerate}
\end{lemma}
\begin{proof}
    \begin{enumerate}
        \item The injectivity follows from that of $f_0$ and $f_1$ and $f_0^"[a]\cap f_1^"[b]=0$ for every $a$ and $b$.
        
        \item Verifying the claim follows from writing down lots of $\Delta_0$ definitions for the functions we used to define the flat pairing.
        First, $s(x)=y$ has the following $\Delta_0$-representation with parameter $\omega$:
        \begin{equation*}
            s(x)=y \iff (x\in\omega\to y=2x+1) \lor (x\notin \omega\to y=x)
        \end{equation*}
        Similarly, we can state $s(x)\in y$ into a $\Delta_0$ formula with parameter $\omega$.
        
        Also, we have
        \begin{itemize}
            \item $b = f_0(a) \iff \forall x\in a [s(x)\in b]\land \forall x\in b \exists y\in a [y=s(x)]$.
            \item $b = f_1(a) \iff \forall x\in a [s(x)\in b]\land 0\in b\land \forall x\in b [x=0\lor \exists y\in a (y=s(x))]$.
            \item $f_0(a)\in b \iff \exists y\in b [y=f_0(a)]$.
            \item $f_1(a)\in b \iff \exists y\in b [y=f_1(a)]$.
        \end{itemize}
        Now let us state the $\Delta_0$-expression for $c=\lceil a,b\rceil$. It breaks into two parts: An expression for $c\subseteq \lceil a,b\rceil$ and an expression for $\lceil a,b\rceil\subseteq c$:
        \begin{itemize}
            \item $c\subseteq\lceil a,b\rceil \iff \forall z\in c [\exists x\in a(z=f_0(x))\lor \exists y\in b (z=f_1(y))]$,
            \item $\lceil a,b\rceil\subseteq c \iff \forall x\in a (f_0(x)\in c)\land \forall y\in b (f_1(y)\in c)$.
        \end{itemize}
        This shows the $\Delta_0$-definability of $a,b\mapsto\lceil a,b\rceil$. Also, we can see that $f_0^{-1}[\lceil a,b\rceil]=a$, $f_1^{-1}[\lceil a,b\rceil]=b$, so inverse images under $f_i$ work as projection functions. Since $f_i$ is $\Delta_0$-definable with parameter $\omega$, so is the map $a\mapsto f_i^{-1}[a]$.
        
        \item First, observe that $\rank s(x)$ is either finite or equal to $\rank x$. Now let us compute a bound for $\rank f_0(a)$: If $\rank a$ is finite, then $\rank s(x)$ for $x\in a$ is also finite, so
        \begin{equation*}
            \rank f_0(a) = \sup\{\rank s(x)+1\mid x\in a\}
        \end{equation*}
        is also finite. If $\rank a$ is infinite, then either $\rank a = \omega$ or there is $x\in a$ such that $\rank x \ge\omega$. In the former case, we have
        \begin{equation*}
            \rank f_0(a) \le \sup\{2\rank x+1\mid x\in a\}\le\omega
        \end{equation*}
        since $\rank x$ is finite for $x\in a$. In the latter case, we have
        \begin{align*}
            \rank f_0(a) &= \sup\{\rank s(x)+1\mid x\in a\land \rank x\ge\omega\} \\
            &=\sup\{\rank x+1\mid x\in a\land \rank x\ge\omega\} \\
            & = \rank a
        \end{align*}
        since $\rank a>\omega$ and $s(x)=x$ if $\rank x\ge\omega$. Hence $\rank f_0(a)$ is either finite or bounded by $\rank a$. Similarly, we can see that $\rank f_1(a)$ is either finite or bounded by $\rank a$.

        Now let us bound $\rank \lceil a,b\rceil$. By a similar argument to the one we gave, we can prove that $\rank f_0^"[a]$ is either finite or at most $\rank a$. Similarly, $\rank f_1^"[b]$ is either finite or at most $\rank b$.
        Hence the rank of $\lceil a,b\rceil = f_0^"[a]\cup f_1^"[b]$ is either finite or at most $\max(\rank a,\rank b)$. \qedhere 
    \end{enumerate}
\end{proof}

We can use the flat pairing function to define a cartesian product of two sets:

\begin{definition}
    For two sets $A$ and $B$, define
    \begin{equation*}
        A \otimes B = \{\lceil x,y\rceil \mid x\in A\land y\in B\}.
    \end{equation*}
\end{definition}
With the rank computation, we have the following:
\begin{lemma} \label{Lemma: Flat product closure}
    Let $\lambda>\omega$ be an ordinal. If $A,B\in V_\lambda$, then $A\otimes B\in V_\lambda$.
\end{lemma}
\begin{proof}
    The lemma is clear when $\lambda$ is limit, so let us consider the case $\lambda=\alpha+1$. Suppose that $A,B\in V_{\alpha+1}$, then $x\in A$ and $y\in B$ imply $x,y\in V_\alpha$. Hence, by \autoref{Lemma: flat pairing properties}, we have $\lceil x,y\rceil \in V_\alpha$. Hence we have
    \begin{equation*}
        A\otimes B =\{\lceil x,y\rceil \in V_\alpha \mid x\in A\land y\in B\} \in V_{\alpha+1}. \qedhere
    \end{equation*}
\end{proof}

We can view binary relations between $A$ and $B$ as subsets of $A\otimes B$ and develop their properties \emph{mutatis mutandis}. However, subsets of $A\otimes B$ and that of $A\times B$ are technically different, so we need different notations for notions about binary relations as subsets of $A\otimes B$. 

\begin{definition}
    $\Dom a = \{x\mid \exists y \lceil x,y\rceil\in a\}$. $\Ran a = \{y\mid \exists x \lceil x,y\rceil\in a\}$. 
\end{definition}

One can also ask if we can define a finite tuple operator under which $V_\lambda$ is closed for every $\lambda>\omega$. A traditional way to define a finite sequence $\lag a_0,\cdots, a_{n-1}\rag$ is by viewing it as a function $i\mapsto a_i$, which is a subset of $n\times \{a_0,\cdots a_{n-1}\}$. If we use Kuratowski ordered pair to formulate $\lag a_0,\cdots, a_{n-1}\rag$, we have
\begin{equation*}
    \rank \lag a_0,\cdots a_{n-1}\rag \le (\max_{i<n}\rank a_i)+3.
\end{equation*}
That is, the operator $a_0,\cdots, a_{n-1}\mapsto \lag a_0,\cdots, a_{n-1}\rag$ may raise the rank by 3. The situation is not dramatically improved even if we use Quine-Rosser ordered pairing as long as we stick to `finite sequences as functions.' We resolve this situation by expanding the definition of $\lceil \cdot,\cdot\rceil$.

Let $s$ be the function we previously mentioned when we defined $\lceil \cdot,\cdot\rceil$. For each $n<\omega$, define $f_n(a) = s^"[a]\cup \{2i \mid i<n\}$. We can also see that $f_n^"[x]$ and $f_m^"[y]$ are mutually disjoint for $n\neq m$, and each $f_n$ is injective. This allows us to code a finite sequence of sets by a single set:
\begin{definition}
    For sets $x_0,\cdots, x_{n-1}$, define $\lceil x_0,\cdots, x_{n-1}\rceil = \bigcup_{k<n} f_k^"[x_k]$.
\end{definition}

By a similar argument we provided in \autoref{Lemma: flat pairing properties}, we get the following:
\begin{lemma} \pushQED{\qed} Let $\lambda>\omega$ be an ordinal.
    \begin{enumerate}
        \item $x_0,\cdots, x_{n-1}\mapsto \lceil x_0,\cdots, x_{n-1}\rceil$ is injective. 
        \item The function $x_0,\cdots, x_{n-1}\mapsto \lceil x_0,\cdots, x_{n-1}\rceil$ and the projection functions are $\Delta_0$-definable over $V_\lambda$ for all $\lambda>\omega$ with parameter $\omega$.
        \item $\rank \lceil x_0,\cdots x_{n-1}\rceil$ is either finite or $\rank \lceil x_0,\cdots x_{n-1}\rceil\le \max(\rank x_0,\cdots,\rank x_{n-1})$.
        \item If $\rank x_0,\cdots, \rank x_{n-1}<\lambda$, then $\rank \lceil x_0,\cdots, x_{n-1}\rceil< \lambda$. \qedhere 
    \end{enumerate}
\end{lemma}

\begin{definition}
    For a set $A$, let $\Fin(A)$ be the set of all flat finite tuples $\lceil a_0,\cdots,a_n\rceil$ such that $a_0,\cdots,a_n\in A$.
\end{definition}

A similar argument for the proof of \autoref{Lemma: Flat product closure} shows the following:
\begin{lemma}\pushQED{\qed}
    Let $\lambda>\omega$ be an ordinal and $A\in V_\lambda$. Then $\Fin(A)\in V_\lambda$.
    \qedhere 
\end{lemma}

\begin{lemma}
    There is a $\Delta_0$-definable function $\lh$ with parameter $\omega$ such that $\lh\lceil a_0,\cdots, a_{n-1}\rceil = n$ for every $a_0,\cdots,a_{n-1}$ and $n$.
\end{lemma}
\begin{proof}
    Let us state an idea about how to define the length function first: If we are given a finite tuple $\lceil x_0,\cdots,x_{n-1}\rceil=\bigcup_{k<n} f_k^"[x_k]$, then elements of $\lceil x_0,\cdots, x_{n-1}\rceil$ take the form $f_k(y)$ for some $k<n$ and $y\in x_k$. Also, $f_k(y) = s^"[y]\cup \{2i\mid i<k\}$. That is, every element of $\lceil x_0,\cdots,x_{n-1}\rceil$ takes the form
    \begin{equation*}
        f_k(y) = s^"[y] \cup \{2i \mid i<k\}
    \end{equation*}
    for some $k<n$. Hence we can track the length of $\lceil x_0,\cdots,x_{n-1}\rceil$ by checking the largest even natural number of elements of $\lceil x_0,\cdots,x_{n-1}\rceil$: That is, we have the following:
    \begin{itemize}
        \item $f_{n-1}^"[x_{n-1}]\subseteq \lceil x_0,\cdots,x_{n-1}\rceil$, and $f_{n-1}(y)\supseteq \{0,2,\cdots,2(n-2)\}$.
        \item No elements of $f_k^"[x_k]$ for $k<n$ contain $2(n-1)$ as its element.
    \end{itemize}
    Thus we get the following: For $n\ge 1$,
    \begin{equation*}
        n-1 = \min \{m<\omega\mid \forall y\in \lceil x_0,\cdots x_{n-1}\rceil (2m\notin y)\}.
    \end{equation*}
    Thus, let us define $\lh\sigma$ by
    \begin{equation*}
        \lh(\sigma) = 
        \begin{cases}
            \min\{m<\omega\mid \forall y\in \sigma(2m\notin y)\}+1 & \text{if $\sigma\neq 0$ and the minimum exists},\\
            0 &\text{otherwise.}
        \end{cases}
    \end{equation*}
    It is easy to see that the formula $\lh(\sigma)=n$ is $\Delta_0$ with parameter $\omega$, and the previous argument shows $\lh\lceil x_0,\cdots,x_{n-1}\rceil = n$ for $n\ge 1$. Also, the empty tuple $\lceil \cdot \rceil$ is the empty set, so $\lh\lceil\cdot\rceil = 0$.
\end{proof}

\begin{lemma}
    There is a $\Delta_0$-definable function $\sigma,\tau\mapsto \sigma^\frown \tau$ with parameter $\omega$ satisfying
    \begin{equation*}
        \lceil a_0,\cdots,a_{m-1}\rceil ^\frown \lceil b_0,\cdots, b_{n-1}\rceil =
        \lceil a_0,\cdots,a_{m-1}, b_0,\cdots, b_{n-1}\rceil 
    \end{equation*}
    for all $a_0,a_1,\cdots,a_{m-1},b_0,\cdots,b_{n-1}$.
\end{lemma}
\begin{proof}
    What we need is to turn the set $\lceil b_0,\cdots,b_{n-1}\rceil = f_0^"[b_0]\cup\cdots\cup f_{n-1}^"[b_{n-1}]$ into the set 
    \begin{equation*}
        f_m^"[b_0]\cup\cdots\cup f_{m+n-1}^"[b_{n-1}].
    \end{equation*}
    To do this, let us define a subsidiary function $\varsigma_m(a)$ satisfying
    \begin{equation} \label{Formula: Subsidiary function condition}
        \varsigma_m(s^"[a]\cup\{2k\mid k<n\}) = s^"[a]\cup \{2k\mid k<n+m\}.
    \end{equation}
    We can see that the following function is $\Delta_0$ with parameter $\omega$:
    \begin{equation*}
        g(a) = 
        \begin{cases}
            n & \text{If $n<\omega$ is the least natural number such that }2n\notin a, \\
            0 & \text{If }\forall n<\omega (2n\in a).
        \end{cases}
    \end{equation*}
    Hence $\varsigma_m(a) :=  a\cup \{2k \mid k<m+g(a)\}$ is also $\Delta_0$-definable with parameter $\omega$, and satisfies \eqref{Formula: Subsidiary function condition}. We can see that
    \begin{equation*}
        \varsigma_m^"[f_0^"[b_0]\cup\cdots\cup f_{n-1}^"[b_{n-1}]] =  f_m^"[b_0]\cup\cdots\cup f_{m+n-1}^"[b_{n-1}],
    \end{equation*}
    holds, so we can consider the following definition:
    \begin{equation*}
        \sigma^\frown\tau := \sigma\cup \varsigma_{\lh(\sigma)}^"[\tau].
    \end{equation*}
    Then $\upsilon=\sigma^\frown\tau$ is $\Delta_0$ with parameter $\omega$ since
    \begin{equation*}
        x\in \sigma^\frown\tau \iff x\in\sigma\lor \exists y\in\tau \exists m<\omega [m=\lh(\sigma)\land x=\varsigma_m(y)],
    \end{equation*}
    is $\Delta_0$, so we can express both of $\upsilon\subseteq \sigma^\frown\tau$ and $\sigma^\frown\tau\subseteq\upsilon$ in $\Delta_0$-formulas with parameter $\omega$.
\end{proof}

Finally, note that although $V_\lambda$ does not have every subset of $V_\lambda$, it can code subsets of $V_\lambda$ of small size, by encoding $\{x_i\mid i\in I\}$ into $\{\lag i,y\rag \mid i\in I\land y\in x_i\}$. Let us introduce the notation for decoding the collection of subsets:
\begin{definition}
    For a binary relation $a\in V_\lambda$ and $i\in\Dom(a)$, define the \emph{$i$th slice of $a$} by
    \begin{equation*}
        (a)_i = \{x\mid \lceil i,x\rceil\in a\}.
    \end{equation*}
    Also, for sets $a,b\in V_\lambda$, define
    \begin{equation*}
        (a:b) = \{(a)_i\mid i\in b\}.
    \end{equation*}
    That is, $(a:b)$ consists of all the slices of $a$ indexed by members of $b$.
\end{definition}

\section{Schlutzenberg's model for a Kunen cardinal} \label{Section: Schlutzenberg model}
Kunen's well-known argument shows that $\ZFC$ is incompatible with an elementary embedding $j\colon V_{\lambda+2}\to V_{\lambda+2}$, but it was open whether it is compatible with $\ZF$. Surprisingly, Schlutzenberg showed that $j\colon V_{\lambda+2}\to V_{\lambda+2}$ is consistent modulo the consistency of $I_0(\lambda)$ as stated in \autoref{Theorem: Schlutzenberg's theorem}. Let us provide a more precise statement for Schlutzenberg's theorem:
\begin{theorem}[Schlutzenberg \cite{Schlutzenberg2020Kunen}] \pushQED{\qed}
    Working over the theory $\ZF + \DC_\lambda + I_0(\lambda)$ with the witnessing embedding $j\colon L(V_{\lambda+1})\to L(V_{\lambda+1})$, the model $L(V_{\lambda+1},k)$ for $k=j\restricts V_{\lambda+2}^{L(V_{\lambda+1})}$ satisfies the following statements:
    \begin{enumerate}
        \item $\ZF + \DC_\lambda + I_0(\lambda)$,
        \item $k\colon V_{\lambda+2} \to V_{\lambda+2}$ is elementary, and
        \item $V_{\lambda+2}\subseteq L(V_{\lambda+1})$. \qedhere 
    \end{enumerate}
\end{theorem}
Here $\DC_\lambda$ (or $\lambda\mhyphen \DC$) is the axiom claiming the following: Suppose that $X$ is a set and let $T\subseteq X^{<\lambda}$ be a tree. If for every $\sigma\in T$ there is $x\in X$ such that $\sigma^\frown \lag x\rag \in T$, then $T$ has a branch of length $\lambda$ in the sense that there is $f\in X^\lambda$ such that $f\restricts \alpha\in T$ for all $\alpha<\lambda$. 
However, $\DC_\lambda$ is not the only choice-like principle that holds over Schlutzenberg's model. The following axiom appears in \cite{Goldberg2021EvenOrdinals}:
\begin{definition}
    We say $V_{\alpha+1}$ satisfies \emph{the Collection Principle} if every binary relation $R\subseteq V_\alpha\times V_{\alpha+1}$ has a subrelation $S\subseteq R$ such that $\dom R=\dom S$ and $\ran S$ is a surjective image of $V_{\alpha+1}$.
\end{definition}
We will see later that the Collection principle for $V_{\alpha+1}$ is associated with the second-order Collection over $H(V_{\alpha+1})$. The following lemma is essential to prove the validity of the Collection Principle for $V_{\lambda+1}$ over Schlutzenberg's model. The proof is included for completeness.
\begin{lemma} \label{Lemma: Global SVC for Schlutzenberg model}
    Schlutzenberg's model $L(V_{\lambda+1},k)$ thinks there is a definable surjection $\Psi\colon \Ord\times V_{\lambda+1}\to L(V_{\lambda+1},k)$.
\end{lemma}
\begin{proof}
    Let us remember the following well-known fact about $L(A)$, whose proof is essentially the same as the proof of the existence of the global well-order over $L$: There is a surjection $\Phi_A\colon \Ord\times A\to L(A)$ definable over $L(A)$ such that for every $\alpha$, $\Phi\restricts (\alpha\times A)\in L(A)$. 
    
    Hence Schlutzenberg's model $L(V_{\lambda+1},k)=L(V_{\lambda+1}\times k)$ has a definable bijection $\Phi_{V_{\lambda+1}\times k}\colon \Ord\times (V_{\lambda+1}\times k)\to L(V_{\lambda+1},k)$.
    Here $k$ is a function with domain $V_{\lambda+2}^{L(V_{\lambda+1})}$, and we have $V_{\lambda+2}^{L(V_{\lambda+1})}\subseteq L(V_{\lambda+1})$. Thus by using the definable surjection $\Phi_{V_{\lambda+1}}\colon \Ord\times V_{\lambda+1}\to L(V_{\lambda+1})$, we get a surjection $\Ord\times V_{\lambda+1}\to k$.
    By composing the previous surjection with $\Phi_{V_{\lambda+1}}\times k$, we get the desired surjection $\Psi$.
\end{proof}

\begin{lemma}
    Working over the Schlutzenberg's model, $V_{\lambda+2}$ satisfies the Collection principle.
\end{lemma}
\begin{proof}
    Let us argue inside Schlutzenberg's model. Let $R\subseteq V_{\lambda+1}\times V_{\lambda+2}$ be a relation. For each $x\in \dom R$, define 
    \begin{equation*}
        \alpha(x) = \min\{\xi\mid \exists a\in V_{\lambda+1} [\lag x, \Psi(\xi,a)\rag\in R]\}.
    \end{equation*}
    Then define
    \begin{equation*}
        S = \{\lag x, \Psi(\alpha(x),a)\rag \mid \lag x, \Psi(\alpha(x),a)\rag\in R\land a\in V_{\lambda+1}\}.
    \end{equation*}
    Clearly $S\subseteq R$ and $\dom R=\dom S$. Since the map $\lag x,a\rag\mapsto \lag x, \Psi(\alpha(x),a)\rag$ is a surjective map from $\dom R\times V_{\lambda+1}$ to $S$, it follows that $S$ is a surjective image of $V_{\lambda+1}$.
\end{proof}

The following definition is due to \cite{GoldbergSchlutzenberg2023}, but in a slightly different form.

\begin{definition}
    Let $M$ be a transitive set closed under flat pairing and $\Dom$. An elementary embedding $j\colon M\to M$ is \emph{cofinal} if for every $a\in M$ we can find $b\in M$ such that $a\in (j(b):\Dom j(b))$.
\end{definition}

This definition should be understood in a second-order set-theoretic context: Let us consider the case $M=V_{\lambda+1}$. Then we may view $M$ as a model $(V_\lambda,V_{\lambda+1})$ of second-order set theory.
Then the cofinality of $j\colon V_{\lambda+1}\to V_{\lambda+1}$ can be understood as follows: For every `class' $a\in V_{\lambda+1}$ over $V_\lambda$ we can find a collection of `classes' $B=\{(b)_x\mid x\in\Dom b\}$ coded by $b\in V_{\lambda+1}$ such that $a\in j(B)$. It is known that every non-trivial elementary embedding from $V_\gamma$ to itself for an even $\gamma$ is cofinal:

\begin{theorem}[Goldberg-Schlutzenberg {\cite[Theorem 3.10]{GoldbergSchlutzenberg2023}}] \pushQED{\qed}
    Let $\gamma$ be an even ordinal, that is, $\gamma$ is of the form $\delta+2n$ for some limit $\delta$ and $n<\omega$. Then if $j\colon V_\gamma\to V_\gamma$ is nontrivial and elementary, $j$ is cofinal. \qedhere
\end{theorem}
In fact, if $j\colon V_\gamma\to V_\gamma$ is cofinal, then $\gamma$ must be even: If $\gamma$ is odd and $j\colon V_\gamma\to V_\gamma$, then $j$ is definable from a parameter in $V_\gamma$. \cite[Theorem 3.11]{GoldbergSchlutzenberg2023} says no cofinal embedding $j\colon V_\gamma\to V_\gamma$ is definable from parameters in $V_\gamma$.

\section{A model for \texorpdfstring{$\ZF^-$}{ZF-} with a cofinal Reinhardt embedding}

In this section, we prove that $H(V_{\lambda+2})$ in Schlutzenberg's model with a Kunen cardinal $\lambda$ is a model for $\ZF^-$ with a cofinal Reinhardt embedding. We approach this problem by coding sets with trees, which is necessary to translate an elementary embedding over $V_{\lambda+2}$ to that over $H(V_{\lambda+2})$. 
\emph{Throughout this section, we work inside Schlutzenberg's model. Also, we always assume every tuple and pair is flat unless specified.}

Let us start with the definition of trees that turn out to code elements of $H(V_{\lambda+2})$:
\begin{definition}
    A set $T$ of finite tuples of sets is a \emph{tree over $X$} if it satisfies the following conditions:
    \begin{enumerate}
        \item Every element of $T$ is a finite flat tuple of elements of $X$.
        \item $T$ is closed under initial segments: That is, if $s\in T$ and if $t$ is an initial segment of $s$, then $t\in T$.
        \item The empty sequence $\lceil \cdot \rceil$ is in $T$.
    \end{enumerate}

    Furthermore, we say $T$ is \emph{suitable} if it satisfies the following additional condition:
    \begin{enumerate} \setcounter{enumi}{3}
        \item $T$ is well-founded, that is, for every non-empty subset $Z\subseteq T$ there is a node $\sigma\in Z$ such that $\sigma^\frown \lceil u\rceil \notin T$ for every $u$.
    \end{enumerate}

    For a suitable $T$ and $\sigma\in T$, define
    \begin{itemize}
        \item $T\downarrow\sigma:=\{\tau\mid \sigma^\frown \tau\in T\}$,\footnote{We assume every tree grows downwards, so $T\downarrow \sigma$ is the subtree of $T$ starting from $\sigma$.} and
        \item $\suc_T(\sigma) = \{u \mid \sigma^\frown \lceil u\rceil\in T\}$.
    \end{itemize}
\end{definition}

For each $a\in H(V_{\lambda+2})$, there is a (not necessarily unique) way to code $a$ into a suitable tree. Observe that if $a\in H(V_{\lambda+2})$ is non-empty, then there is an onto function $f\colon V_{\lambda+1}\to \trcl(a)$. Then consider the following coding tree:
\begin{definition}
    For each $a\in H(V_{\lambda+2})$ and an onto map $f\colon X\to \trcl(a)$ for some $X\subseteq V_{\lambda+1}$, define the \emph{canonical tree} $C_{a,f}$ by the set of all finite tuples $\lceil u_0,\cdots, u_{n-1}\rceil\in \Fin(X)$ such that $a\ni f(u_0)\ni\cdots\ni f(u_{n-1})$.
\end{definition}
Since $\in$ is well-founded, the canonical tree $C_{a,f}$ is also well-founded, so it is a suitable tree over $V_{\lambda+1}$.
The reason for allowing the domain of $f$ to be an arbitrary subset of $V_{\lambda+1}$ is technical: One of the reasons is that it allows us a canonical tree coding the empty set with $\dom f = 0$. The role of the technical requirement will become evident in the proof of \autoref{Lemma: Canonical tree codes the set}.

Conversely, we also have a tool to decode a suitable tree into a set:
\begin{definition}
    Let $T$ be a suitable tree. Define the \emph{transitive collapse $\Tcoll(T)$ of $T$} by recursion on $T$:
    \begin{equation*}
        \Tcoll(T) = \{\Tcoll(T\downarrow \lceil u\rceil)\mid\lceil u\rceil\in T\}.
    \end{equation*}
\end{definition}

The following lemma shows that $C_{a,f}$ codes $a$:
\begin{lemma} \label{Lemma: Canonical tree codes the set}
    Let $a\in H(V_{\lambda+2})$ and $f\colon X \to \trcl\{a\}$ be an onto map for some $X\subseteq V_{\lambda+1}$. Then $C_{a,f}$ is a suitable tree over $V_\lambda$ and $\Tcoll(C_{a,f})=a$.
\end{lemma}
\begin{proof}
    Let us prove it by induction on $a \in H(V_{\lambda+2})$. Suppose that for every non-empty $x\in a$ and an onto map $g\colon Y \to \trcl(x)$ for $Y\subseteq V_{\lambda+1}$ we have $\Tcoll(C_{x,g})=x$. Fix $X\subseteq V_{\lambda+1}$ and an onto map $f\colon X\to \trcl(a)$. Then we have
    \begin{equation*}
        \Tcoll(C_{a,f}) = \{\Tcoll(C_{a,f}\downarrow\lceil u\rceil)\mid u\in V_{\lambda+1}\}.
    \end{equation*}
    We claim first that $\Tcoll(C_{a,f})\subseteq a$: If $u_0$ is an immediate successor of the empty node over $T$, then
    \begin{equation*}
        C_{a,f}\downarrow\lceil u_0\rceil = \{ \lceil u_1,\cdots,u_{n-1}\rceil \in \Fin(X) \mid f(u_0)\ni f(u_1)\ni \cdots f(u_{n-1}) \}.
    \end{equation*}
    Take $X_0 = \{u\in X \mid f(u)\in \trcl(f(u_0))\}$. Then $f\restriction  X_0$ is onto $\trcl(f(u_0))$, and we have
    \begin{equation*}
        C_{a,f} \downarrow\lceil u_0\rceil = C_{f(u_0),f\restriction  X_0}.
    \end{equation*}
    By the inductive hypothesis on $f(u_0)\in a$, we have
    \begin{equation*}
        \Tcoll(C_{a,f}\downarrow\lceil u_0\rceil) = \Tcoll(C_{f(u_0),f\restriction  X_0})=f(u_0).
    \end{equation*}
    This shows $\Tcoll(C_{a,f}) = \{f(u)\mid \lceil u\rceil\in C_{a,f}\} = \{f(u)\mid u\in X, f(u)\in a\}=a$.
\end{proof}

The following notion allows us to formulate the equality and the membership relationship between suitable trees in a definable manner over $V_{\lambda+2}$: 
\begin{definition}
    Let $S$ and $T$ be suitable trees over $V_{\lambda+1}$. We say $X$ is a \emph{multi-valued isomorphism from $S$ to $T$} if $X\subseteq S\otimes T$ satisfies the following: 
    \begin{enumerate}
        \item $\lceil \lceil \cdot \rceil, \lceil \cdot \rceil \rceil \in X$.
        \item $\lceil \sigma,\tau\rceil\in X$ iff the follwing holds:
        \begin{enumerate}
            \item $\forall u\in \suc_S(\sigma) \exists v\in \suc_T(\tau) [\lceil \sigma^\frown\lceil u\rceil, \tau^\frown \lceil v\rceil\rceil\in X]$.
            \item $\forall u\in \suc_T(\tau) \exists v\in \suc_S(\sigma)  [\lceil \sigma^\frown\lceil u\rceil, \tau^\frown \lceil v\rceil\rceil\in X]$.
        \end{enumerate}
    \end{enumerate}
    Then define
    \begin{enumerate}
        \item $S=^* T \iff \exists X (\text{There is a multi-valued isomorphism from $S$ to $T$})$.
        \item $S\in^* T \iff \exists u\in \suc_T(\lceil \cdot \rceil) \exists X (S =^* T\downarrow \lceil u\rceil)$.
    \end{enumerate}
\end{definition}

\begin{lemma} \label{Lemma: Suitable tree coding for relations}
    For suitable trees $S$ and $T$, we have the following:
    \begin{enumerate}
        \item $S=^* T$ iff $\Tcoll(S)=\Tcoll(T)$.
        \item $S\in^* T$ iff $\Tcoll(S)\in \Tcoll(T)$.
    \end{enumerate}
\end{lemma}
\begin{proof}
    Suppose that we have $S=^* T$, and let $X\subseteq S\otimes T$ be a multi-valued isomorphism between $S$ and $T$. We prove the following by induction on $\lceil\sigma,\tau\rceil\in S\otimes T$ with the pointwise comparison relation that is well-founded:
    \begin{equation*}
        \lceil \sigma,\tau\rceil \in X\implies \Tcoll(S\downarrow\sigma)=\Tcoll(T\downarrow\tau).
    \end{equation*}
    Suppose that the above implication holds for every $\lceil \sigma^\frown\lceil u\rceil, \tau^\frown \lceil v\rceil \rceil$ such that $ \sigma^\frown\lceil u\rceil\in S$ and $\tau^\frown \lceil v\rceil\in T$.
    Then we have
    \begin{equation*}
        \Tcoll(S\downarrow\sigma) = \{\Tcoll(S\downarrow\sigma^\frown\lceil u\rceil) \mid u\in \suc_S(\sigma)\}.
    \end{equation*}
    Since $X$ is a multi-valued isomorphism from $S$ to $T$, we have for each $u\in \suc_S(\sigma)$ there is $v\in \suc_T(\tau)$ such that $\lceil \sigma^\frown\lceil u\rceil, \tau^\frown \lceil v\rceil\rceil\in X$. By the inductive hypothesis, we get
    \begin{equation*}
        \Tcoll(S\downarrow \sigma^\frown\lceil u\rceil) = \Tcoll(T\downarrow\tau^\frown\lceil v\rceil)
    \end{equation*}
    This implies $\Tcoll(S\downarrow\sigma)\subseteq \Tcoll(T\downarrow\tau)$, and the symmetric argument gives the reversed inclusion. Hence, we get the desired equality.
    The implication from $S=^*T$ to $\Tcoll(S)=\Tcoll(T)$ follows from the above lemma by letting $\sigma=\tau=\lceil \cdot \rceil$.

    Conversely, assume that we have $\Tcoll(S)=\Tcoll(T)$. Define
    \begin{equation*}
        X = \{\lceil \sigma,\tau\rceil \mid \Tcoll(S\downarrow\sigma) = \Tcoll(T\downarrow\tau)\}.
    \end{equation*}
    We claim that $X$ is a multi-valued isomorphism from $S$ to $T$. $\lceil \lceil \cdot \rceil, \lceil \cdot \rceil \rceil \in X$ is clear, and let us prove the remaining conditions.

    Suppose that $\lceil \sigma,\tau\rceil \in X$ and $u\in \suc_S(\sigma)$. We want to find $v\in \suc_T(\tau)$ such that $\lceil \sigma^\frown\lceil u\rceil, \tau^\frown \lceil v\rceil\rceil\in X$, or equivalently,
    \begin{equation*}
        \Tcoll(S\downarrow(\sigma^\frown \lceil u\rceil)) =  \Tcoll(T\downarrow(\tau^\frown \lceil v\rceil)).
    \end{equation*}
    However, $\lceil \sigma,\tau\rceil \in X$ means $\Tcoll(S\downarrow\sigma)=\Tcoll(T\downarrow \tau)$, so we can find a desired $v$. Proving the remaining subcondition is identical.

    Conversely, assume that we have the following properties:
    \begin{enumerate}
        \item $\forall u\in \suc_S(\sigma) \exists v\in \suc_T(\tau) [\lceil \sigma^\frown\lceil u\rceil, \tau^\frown \lceil v\rceil\rceil\in X]$.
        \item $\forall u\in \suc_T(\tau) \exists v\in \suc_S(\sigma)  [\lceil \sigma^\frown\lceil u\rceil, \tau^\frown \lceil v\rceil\rceil\in X]$.
    \end{enumerate}
    The first condition implies $\Tcoll(S\downarrow \sigma)\subseteq \Tcoll(T\downarrow\tau)$, and the remaining one implies the reversed inclusion, so $\lceil\sigma,\tau\rceil\in X$.

    Proving the equivalence between $S\in^* T$ and $\Tcoll(S)\in \Tcoll(T)$ is immediate, so we omit its proof.
\end{proof}

\begin{lemma} \label{Lemma: Definability of notions for suitable trees}
    The following propositions are all definable over $V_{\lambda+2}$ with parameter $V_{\lambda+1}$:
    \begin{enumerate}
        \item $T$ is a suitable tree over $V_{\lambda+1}$.
        \item $S=^* T$.
        \item $S \in^* T$
    \end{enumerate}
\end{lemma}
\begin{proof}
    Let us recall that the statement `$T$ is a tree over $V_{\lambda+1}$' is a first-order statement over $V_{\lambda+2}$ with parameter $V_{\lambda+1}$. Furthermore, the well-foundedness of a tree $T$ over $V_{\lambda+1}$ is first-order expressible over $V_{\lambda+2}$ since the well-foundedness of a tree $T$ over $V_{\lambda+1}$ is equivalent to
    \begin{equation*}
       \forall X\in V_{\lambda+2}[ X\subseteq T \land X\neq 0\to \exists \sigma\in X \forall u\in V_{\lambda+1} (\sigma^\frown \lceil u\rceil\notin X)].
    \end{equation*}
    
    It shows the first-order expressibility of being a suitable tree over $V_{\lambda+1}$ in the structure $V_{\lambda+2}$. The first-order expressibility of the remaining two in $V_{\lambda+2}$ is clear since if there is a multi-valued isomorphism $X \subseteq S\otimes T$ between two trees $S$ and $T$ over $V_{\lambda+1}$, then $X\in V_{\lambda+2}$.
\end{proof}

Let us give an example of constructing a new suitable tree from an old one, whose construction plays a critical role in the next lemma:
\begin{example} \label{Example: Suitable tree for TC}
    Let $T$ be a suitable tree over $V_{\lambda+1}$, and consider the following tree:
    \begin{equation*}
        T' = \{\lceil \cdot \rceil\}\cup \{\lceil \sigma\rceil^\frown \tau\mid \sigma^\frown\tau\in T\}.
    \end{equation*}
    We claim that $\Tcoll(T')$ is the transitive closure of $\Tcoll(T)$. First, we have
    \begin{align*}
        \Tcoll(T') & = \{\Tcoll(T')\downarrow\lceil \sigma\rceil :\sigma\in T\} \\ &= \{\Tcoll(T\downarrow\sigma)\mid \sigma\in T\}.
    \end{align*}
    It is clear that $\Tcoll(T)\subseteq \Tcoll(T')$. Also, if $x\in \Tcoll(T')$, then $x=\Tcoll(T\downarrow\sigma)$ for some $\sigma\in T$, and if $y\in x$, then $y=\Tcoll(T\downarrow\sigma^\frown\lceil a\rceil)$ for some $a\in \suc_T(\sigma)$, so $y\in \Tcoll(T')$. This shows that $\Tcoll(T')$ is transitive.

    To see $\Tcoll(T')$ is a transitive closure of $\Tcoll(T)$, we show that $\Tcoll(T')$ is contained in every transitive set containing $\Tcoll(T)$: Suppose that $M$ is a transitive set such that $M\supseteq \Tcoll(T)$. Then by the induction of the length of $\sigma$, we can prove $\Tcoll(T\downarrow\sigma)\in M$ for every $\sigma\in T$.
\end{example}

Now we can see that every suitable tree over $V_{\lambda+1}$ codes a set in $H(V_{\lambda+2})$:
\begin{lemma}
    Let $T$ be a suitable tree over $V_{\lambda+1}$. Then $\Tcoll(T)\in H(V_{\lambda+2})$
\end{lemma}
\begin{proof}
    Consider the map $f$ of domain $T$ given by $f(\sigma)=\Tcoll(T\downarrow\sigma)$, and let $T'$ be a new suitable tree over $V_{\lambda+1}$ defined in \autoref{Example: Suitable tree for TC}. Since
    \begin{equation*}
        \Tcoll(T') = \{\Tcoll(T\downarrow\sigma)\mid \sigma\in T\} = \trcl(\Tcoll(T)),
    \end{equation*}
    $f$ is an onto map from $T\in V_{\lambda+2}$ to $\trcl(\Tcoll(T))$. That is, $\trcl(\Tcoll(T))$ is an image of a member of $V_{\lambda+2}$, so $\trcl(\Tcoll(T))\subseteq H(V_{\lambda+2})$. We then get the desired result by considering a new tree
    \begin{equation*}
        T'' = \{\lceil \cdot \rceil\} \cup \{\lceil 0\rceil^\frown \sigma\mid\sigma\in T\}
    \end{equation*}
    obtained from $T$ by adding an extra top element to it.
\end{proof}

Conversely, every set in $H(V_{\lambda+2})$ is coded by a suitable tree over $V_{\lambda+1}$:
\begin{lemma} \label{Lemma: Every element of H is coded by a suitable tree}
    Assume that the Collection Principle for $V_{\lambda+1}$ holds. If $a\in H(V_{\lambda+2})$, then there is a suitable tree $T\in V_{\lambda+2}$ over $V_{\lambda+1}$ such that $\Tcoll(T)=a$.
\end{lemma}
\begin{proof}
    We prove it by $\in$-induction on $a\in H(V_{\lambda+2})$. The case $a=0$ is trivial, so let us assume that $a$ is not empty. Suppose that for each $b\in a$ there is a suitable tree $t\in V_{\lambda+2}$ over $V_{\lambda+1}$ such that $\Tcoll(t)=b$. We cannot `choose' such a tree for each $b\in a$ since we do not have an appropriate form of the axiom of choice. We appeal to the Collection principle to overcome this issue.
    
    Since $a\in H(V_{\lambda+2})$, there is an onto map $f\colon V_{\lambda+1}\to a$. Now consider the relation $R\subseteq V_{\lambda+1}\times V_{\lambda+2}$ given by
    \begin{equation*}
        R = \{\lag x,t\rag \mid \Tcoll(t)=f(x) \}.
    \end{equation*}
    It is clear that $\dom R = V_{\lambda+1}$. By the Collection Principle for $V_{\lambda+1}$, there is a subrelation $S\subseteq R$ such that $\dom R=\dom S=V_{\lambda+1}$ and $\ran S$ is a surjective image of $V_{\lambda+1}$. Now let us construct a new tree $T$ as follows:
    \begin{equation*}
        T = \{\lceil \cdot \rceil\}\cup \{\lceil x\rceil^\frown \sigma\mid u\in V_{\lambda+1} \land \exists t [\lag x,t\rag\in S\land \sigma\in t]\}.
    \end{equation*}
    It remains to see that $\Tcoll(T)=a$, which follows from the following computation:
    \begin{align*}
        \Tcoll(T) &= \{\Tcoll(T\downarrow\lceil x\rceil) \mid x\in\dom S\} \\
        &= \{\Tcoll(t) \mid x\in \dom S \land (x,t)\in S\} = \{f(x) \mid x\in \dom S\} = a.
        \qedhere 
    \end{align*}
\end{proof}

Now we check that $H(V_{\lambda+1})$ is a model of Second-order $\ZF^-$. Verifying all axioms other than Second-order Collection follows from \autoref{Corollary: H(X) models SO Zermelo}:
\begin{proposition}
    For $\alpha>\omega$, $H(V_\alpha)$ models Second-order Zermelo set theory without Powerset and Choice. Also, $V_{\alpha}\subseteq H(V_\alpha)$ for every $\alpha\ge\omega$.
\end{proposition}
\begin{proof}
    It suffices to show by \autoref{Corollary: H(X) models SO Zermelo} that $V_\alpha$ is closed under disjoint union. For $a,b\in V_\alpha$, define
    \begin{equation*}
        a \sqcup b := \{\lceil 0,x\rceil \mid x\in a\} \cup \{\lceil 1,y\rceil \mid y\in b\}.
    \end{equation*}
    Then $a\sqcup b\in V_\alpha$, and the map $\iota^i_{a,b}(x) = \lceil i,x\rceil$ witnesses $\sqcup$ is a disjoint union.
    For the remaining part, if $\alpha$ is limit, then $V_\beta \in V_\alpha$ for every $\beta<\alpha$ implies $V_\alpha\subseteq H(V_\alpha)$.
    If $\alpha=\gamma+1$, then $V_\gamma\subseteq H(V_\alpha)$. For each $A\in V_\alpha$, $V_\gamma\cup \{A\}$ is a transitive set, and the function $f\colon V_\gamma \to V_\gamma \cup \{A\}$ defined by
    \begin{equation*}
        f(x) = 
        \begin{cases}
            y & \text{if } x= \lceil 0,y\rceil \text{ for some $y\in V_\gamma$,}\\
            A & \text{if } x = \lceil 1,0\rceil, \\
            0 & \text{otherwise}
        \end{cases}
    \end{equation*}
    is a surjection. This shows $A\in H(V_\alpha)$ for every $A\subseteq V_\gamma$, so $V_\alpha\subseteq H(V_\alpha)$.
\end{proof}

$H(V_\alpha)$ does not satisfy the Second-order Collection in general. 
However, the Collection Principle for $V_{\lambda+1}$ allows us to have Collection over $H(V_{\lambda+2})$:
\begin{proposition}
    If the Collection Principle for $V_{\lambda+1}$ holds, then $H(V_{\lambda+2})$ satisfies the second-order Collection.
\end{proposition}
\begin{proof}
    Fix $a\in H(V_{\lambda+2})$ and let $R\subseteq a\times H(V_{\lambda+2})$ be a binary relation such that $\dom R = a$.
    Define
    \begin{equation*}
        \tilde{R} = \{\lag x,t \rag \in a\times V_{\lambda+2}\mid \text{$t$ is a suitable tree over $V_{\lambda+1}$ and } \lag x,\Tcoll(t)\rag\in R\}.
    \end{equation*}
    Then $\tilde{R}\subseteq a\times V_{\lambda+2}$ is a binary relation with domain $a$. Hence by the Collection principle for $V_{\lambda+1}$, we have a subrelation $\tilde{S}\subseteq \tilde{R}$ with a surjection $f\colon V_{\lambda+1}\to \tilde{S}$. Now take
    \begin{equation*}
        b = \{\Tcoll(t) \mid t\in \ran \tilde{S}\}.
    \end{equation*}
    Clearly we have $\forall x\in a\exists y\in b (\lag a,b\rag\in R)$. To see $b\in H(V_{\lambda+2})$, consider the following suitable tree coding $b$:
    \begin{equation*}
        \tilde{t} = \{\lceil \cdot \rceil\}\cup \{\lceil x\rceil^\frown \sigma \mid \lag x,t\rag\in\tilde{S} \land \sigma\in t\}.
    \end{equation*}
    It is easy to see that $\tilde{t}$ is a suitable tree over $V_{\lambda+1}$ and $\Tcoll(\tilde{t}) = b$, so $b\in H(V_{\lambda+2})$.
\end{proof}

\begin{definition}
    Let $k\colon V_{\lambda+2}\to V_{\lambda+2}$ be an elementary embedding. We define $j\colon H(V_{\lambda+2})\to H(V_{\lambda+2})$ as follows:
    For a suitable tree $T$ over $V_{\lambda+1}$ such that $a=\Tcoll(T)$, define $j(a) := \Tcoll(k(T))$.
\end{definition}
We need to ensure the above definition works, but it is almost immediate from the elementarity of $k$: Suppose that $a=\Tcoll(S)=\Tcoll(T)$ for suitable trees over $V_{\lambda+1}$. Then by \autoref{Lemma: Suitable tree coding for relations}, $S=^* T$. Since $=^*$ is definable over $V_{\lambda+2}$ by \autoref{Lemma: Definability of notions for suitable trees}, we have $k(S)=^* k(T)$. Again by \autoref{Lemma: Suitable tree coding for relations}, we get $\Tcoll(k(S))=\Tcoll(k(T))$.

\begin{proposition}
   $j\colon H(V_{\lambda+2})\to H(V_{\lambda+2})$ is an elementary embedding.
\end{proposition}
\begin{proof}
    To prove the elementarity of $j$, we first introduce a way to translate a formula over $H(V_{\lambda+2})$ to that over $V_{\lambda+2}$:
    
    \begin{lemma} \label{Lemma: Coding H statements to V statements}
        Let $\phi$ be a formula with $n$ free variables. Then we can find a formula $\phi^\frakt$ satisfying the following:
        Let $t_0,\cdots,t_{n-1}$ be suitable trees over $V_{\lambda+1}$. Then
        \begin{equation*}
            H(V_{\lambda+2})\vDash \phi(\Tcoll(t_0),\cdots,\Tcoll(t_{n-1})) \iff V_{\lambda+2}\vDash \phi^\frakt(t_0,\cdots,t_{n-1}).
        \end{equation*}
    \end{lemma}
    \begin{proof}
        For a given formula $\phi$, let us recursively define $\phi^\frakt$ as follows:
        \begin{enumerate}
            \item $(x=y)^\frakt \equiv (x=^*y)$, $(x\in y)^\frakt \equiv (x\in^* y)$.
            \item If $\circ$ is a logical connective (like $\land$, $\lor$, $\to$), then $(\phi\circ \psi)^\frakt \equiv (\phi^\frakt \circ \psi^\frakt)$.
            \item $(\lnot\phi)^\frakt \equiv \lnot\phi^\frakt$.
            \item If $\sfQ$ is a quantifier, then
            \begin{equation*}
                (\mathsf{Q} x \phi(x))^{\mathfrak{t}} \equiv \mathsf{Q} x  [\text{$x$ is a suitable tree over $V_{\lambda+1}$} \to \phi^{\mathfrak{t}} (x)].
            \end{equation*}
        \end{enumerate}
        We claim that $\phi^\frakt$ satisfies the desired property.
        The atomic case follows from \autoref{Lemma: Suitable tree coding for relations}, and the logical connective cases are trivial.
        For the quantifier case, let us only consider the case $\sfQ = \exists$ since the remaining case is similar. 

        Suppose that we have $H(V_{\lambda+2})\vDash  \exists a \phi(a,\Tcoll(t_0),\cdots,\Tcoll(t_{n-1}))$. By \autoref{Lemma: Every element of H is coded by a suitable tree}, there is a suitable tree $t$ over $V_{\lambda+1}$ such that $\Tcoll(t)=a$. Then apply the inductive hypothesis to get
        \begin{equation*}
            H(V_{\lambda+2})\vDash \phi(\Tcoll(t),\Tcoll(t_0),\cdots,\Tcoll(t_{n-1})) \iff V_{\lambda+2}\vDash\phi^\frakt (t,t_0,\cdots,t_{n-1}).
        \end{equation*}
        Thus we have
        \begin{equation*}
            V_{\lambda+2} \vDash \exists t [\text{$t$ is a suitable tree over $V_{\lambda+1}$}\land \phi^\frakt (t,t_0,\cdots,t_{n-1})],
        \end{equation*}
        which is $V_{\lambda+2}\vDash (\exists y \phi(y,x_0,\cdots,x_{n-1}))^\frakt (t_0,\cdots,t_{n-1})$.
        The converse direction is easy, so we omit it. 
    \end{proof}
    Now let us prove the elementarity of $j$: Suppose that we have $H(V_{\lambda+2})\vDash \phi(a_0,\cdots,a_{n-1})$. By \autoref{Lemma: Every element of H is coded by a suitable tree}, we have suitable trees $t_0,\cdots,t_{n-1}$ over $V_{\lambda+1}$ such that $a_i=\Tcoll(t_i)$ for $i<n$. Hence by \autoref{Lemma: Coding H statements to V statements}, we have $V_{\lambda+2} \vDash \phi^\frakt (t_0,\cdots,t_{n-1})$. By elementarity of $k$, we get $V_{\lambda+2} \vDash \phi^\frakt (k(t_0),\cdots,k(t_{n-1}))$, and the combination of \autoref{Lemma: Coding H statements to V statements} and the definition of $j$ gives $H(V_{\lambda+2})\vDash \phi(j(a_0),\cdots, j(a_{n-1}))$.
\end{proof}

It is immediate that if $k$ is non-trivial, then $j$ is also non-trivial. Also, we have
\begin{theorem}
    $j\colon H(V_{\lambda+2})\to H(V_{\lambda+2})$ is cofinal.
\end{theorem}
\begin{proof}
    Let $a\in H(V_{\lambda+2})$ and $t$ be a suitable tree over $V_{\lambda+1}$ such that $a=\Tcoll(t)$. We want to find $b\in H(V_{\lambda+2})$ such that $a\in j(b)$. To prove this, it suffices to find a suitable tree $t'$ over $V_{\lambda+1}$ such that $t\in^* k(t')$.

    Recall that $k\colon V_{\lambda+2}\to V_{\lambda+2}$ is cofinal, so we have a set $c_0\in V_{\lambda+2}$ such that $t\in (k(c_0):\Dom k(c_0))$. Now let
    \begin{equation*}
        c = \{ \lceil x,\sigma\rceil\in c_0\mid \text{$(c_0)_x$ is a suitable tree over $V_{\lambda+1}$}\} \in V_{\lambda+2}.
    \end{equation*}
    Then define $t'$ as follows:
    \begin{equation*}
        t' = \{\lceil \cdot \rceil\}\cup \{\lceil x\rceil^\frown \sigma\mid  \lceil x,\sigma\rceil\in c\}.
    \end{equation*}
    That is, $t'$ is obtained from $c$ by joining all suitable trees of the form $(c)_x$ for $x\in \Dom c \subseteq \Dom c_0$. Hence $t'$ is a suitable tree over $V_{\lambda+1}$.
    
    Now we claim that $t\in^* k(t')$.
    We can easily see that $t\in (k(c):\Dom k(c))$, so there is $x\in \Dom k(c)$ such that $t=(k(c))_x$. From this, we have
    \begin{equation*}
        k(t')\downarrow \lceil x\rceil = \{\sigma\mid \lceil x,\sigma\rceil\in k(c)\} = (k(c))_x = t.
    \end{equation*}
    Hence we get $t\in^* k(t')$.
\end{proof}

In sum, we get the following:
\begin{corollary}
    Working over the Schlutzenberg's model, if $\lambda$ is a Kunen cardinal, then $H(V_{\lambda+2})$ is a model of $\ZF^-_j + \DC_\lambda$ with a non-trivial cofinal elementary embedding $j\colon V\to V$. Furthermore, $V_{\lambda+1}\in H(V_{\lambda+2})$.
\end{corollary}
\begin{proof}
    We have proved that $H(V_{\lambda+2})$ is a model of second-order $\ZF^-$ with a non-trivial cofinal Reinhardt embedding. $V_{\lambda+1}\in H(V_{\lambda+2})$ is immediate since the canonical tree $C_{V_{\lambda+1}, \mathsf{Id}}$ with the identity map $\mathsf{Id}\colon V_{\lambda+1}\to V_{\lambda+1}$ is a suitable tree over $V_{\lambda+1}$.
    It remains to see that $H(V_{\lambda+2})$ satisfies $\DC_\lambda$, but it immediately follows from $\DC_\lambda$ over Schlutzenberg's model and the second-order Collection over $H(V_{\lambda+2})$.
\end{proof}

\section{Discussions}
Our model for $\ZF^-_j$ with a cofinal Reinhardt embedding $j\colon V\to V$ thinks $V_{\lambda+1}$ exists, and so we may ask if the existence of $V_{\lambda+1}$ is a consequence of the cofinality of a Reinhardt embedding:
\begin{question}
    Working over $\ZF^-_j$ with a non-trivial Reinhardt embedding $j\colon V\to V$, let $\lambda$ be the supremum of the critical sequence $\lag j^n(\crit j)\mid n<\omega\rag$. Can we prove the existence of $\lambda^+$ or $V_{\lambda+1}$? What happens if we assume $\DC_\lambda$?
\end{question}
It might be possible that a construction in \cite{GitmanMatthews2022} may work, but it requires examining the Gitman-Matthews construction without choice. 
Also, we constructed a model of $\ZF^-_j$ with a cofinal Reinhardt embedding $j\colon V\to V$ from a stronger large cardinal axiom than what Matthews \cite{Matthews2020} used. Thus, we may ask if we could get a model of $\ZF^-_j$ with a cofinal Reinhardt embedding.
More generally, we may ask the consistency strength of $\ZF^-_j$ with a cofinal elementary embedding:
\begin{question}
    What is the consistency strength of $\ZF^-_j$ with a cofinal Reinhardt embedding? For example, does it imply the consistency of $\ZFC + I_1$?
\end{question}
The author conjectures $\ZFC+I_1$ does not prove the consistency of the theory ``$\ZF^-_j$ + $j\colon V\to V$ is cofinal + $\DC_\lambda$ + $V_{\lambda+1}$ exists'';  
The author conjectures Laver's argument in \cite{Laver1997Implications} carries over $\ZF^-_j$ with a cofinal $j\colon V\to V$ with extra assumptions $\DC_\lambda$ and the existence of $V_{\lambda+1}$, where $\lambda = \sup_{n<\omega} j^n(\crit j)$.
In particular, we should be able to imply various large cardinal assertions like the existence of an $I_1$-cardinal below $\lambda$ from the previously mentioned theory.

\printbibliography

\end{document}